\documentclass[11pt,a4paper]{article}

\usepackage[dvips]{graphicx}
\usepackage{amssymb,latexsym}
\usepackage{amsmath}
\usepackage{mathrsfs}
\usepackage{psfrag}
\usepackage[all]{xy}
\usepackage{color}
\usepackage{verbatim}
\usepackage{soul}

\newtheorem{prop}{Proposition}[section]
\newtheorem{theo}[prop]{Theorem}
\newtheorem{cor}[prop]{Corollary}

\newtheorem{exs}[prop]{Examples}
\newtheorem{defn}[prop]{Definition}
\newtheorem{defns}[prop]{Definitions}
\newtheorem{rem}[prop]{Remark}
\newtheorem{rems}[prop]{Remarks}

\newenvironment{proof}
{\begin{trivlist} \item[\hskip \labelsep {\bf Proof}\hspace*{3 mm}]}
	{\hfill$\Box$\end{trivlist}}
\newenvironment{acknow}
{\begin{trivlist} \item[\hskip \labelsep {\bf Acknowledgments.}]}
	{\end{trivlist}}

\makeatletter
\newcommand{\tpitchfork}{%
  \vbox{
    \baselineskip\z@skip
    \lineskip-.52ex
    \lineskiplimit\maxdimen
    \m@th
    \ialign{##\crcr\hidewidth\smash{$-$}\hidewidth\crcr$\pitchfork$\crcr}
  }%
}
\makeatother

\def \BR {\mathbb R}
\def \BC {\mathbb C}
\def \BL {\mathbb L}
\def \BK {\mathbb K}
\def \BP {\mathbb P}
\def \BS {\mathbb S}

\def \sp {\rm sp}
\def \a {\alpha}
\def \b {\beta}

\def \d {\delta}

\def \m {\mathfrak{m}}

\begin{document}

\title{On the affine geometry of congruences of lines}
\author{J. W. Bruce and F. Tari}

\maketitle
\begin{abstract}
Congruences, or $2$-parameter families of lines in $3$-space are of interest in many situations, in particular in geometric optics. In this paper we consider elements of their geometry which are invariant under affine changes of co-ordinates, for example that associated with their focal sets, and less well studied focal planes.  We use tools from singularity theory to describe some generic phenomena. In particular we determine the generic singularities of various surfaces in affine 3-space associated to these congruences.
We identify a projective quadric in the projectivised tangent space to the manifold of lines which plays a key role in understanding the affine geometry of ruled surfaces, congruences and $3$-parameter families or complexes.  
Many of the results generalise to lines in $\BR^n, n>3$. 
\end{abstract}

\renewcommand{\thefootnote}{\fnsymbol{footnote}}

\section{Introduction} \label{sec:prel}

Families of lines in Euclidean 3-space $\BR^3$ arise naturally in a variety of ways. Their study has a long history, with contributions from Dupin, Hamilton, Darboux, Blaschke and many others, and they have extensive applications in the sciences and engineering. These families can be viewed as submanifolds of the space of lines, and one can also consider the union of these lines in $\BR^3$. One parameter families of lines are the classical ruled surfaces, with developable surfaces a particularly interesting subclass. Their differential geometry as well as their singularities are well studied; see for example \cite{IzumiyaRuled}. Two parameter families of lines, or congruences of lines in $\BR^3$ have a rich and interesting history and accounts can be found in classic textbooks such as \cite{blaschke, weatherburn}.  Classical differential geometry is largely concerned with linearisation, with a focus on quadratic forms. Generic differential geometry (see for example \cite{wall}) uses singularity theory to identify higher order phenomena and that is what we do here, focusing on affine invariant geometry. The real projective case, with the lines determined by points of the Klein quadric, is very similar, but slightly simpler; we work through a duality result in this case to illustrate the differences. 
The results in the paper are as follows. 
\\
\S\ref{sec:Representation_of_Lines} We give a representation of the space of lines $\BL$ in an affine setting.
\\
\S\ref{sec:structuresTgbundle} We show that the set of lines meeting a line $L$, denoted by $I(L)$ has a cone singularity at $L$ and prove that its tangent cone at $L$ is a non-singular quadric $Q_L$ in the projectivise tangent space to $\mathbb L$ at $L$. The quadric $Q_L$ plays a key role in the affine geometry of submanifolds of $\BL$.
\\
\S\ref{sec:complexcase} Every congruence gives rise to a {\it middle set}, in the affine case most naturally defined as the midpoint of pairs of focal points. The result here is used to show that these midpoints are real even if the focal points are complex. 
\\
\S\ref{sec:geomCong} This section introduces congruences and shows how their geometry at each point is reflected in the intersection of the projective line determined by the tangent plane with the quadric $Q_L$. There is  a natural binary differential equation on the congruence' this and some other classical BDEs linked to congruences are studied in \cite{brucetariBDE}, see also \cite{CraizerGarcia}.
\\
\S\ref{sec:Focalsets} The principle invariants of a congruence are the focal sets (where the set of lines focus) and their focal planes. These are studied via incidence sets, and a simple transversality result.
\\
\S\ref{sec:transversality} We establish a second transversality result that yields some of our normal forms via unfolding theory.
\\
\S\ref{sec:ContactModels} One tool used in the generic geometry of surfaces in $3$-space is their contact with \lq model\rq\ submanifolds. A range of models are presented and used to study generic congruences. 
\\
\S\ref{sec:midSurf} We finish by describing the generic singularities of the middle surface.

\section{A representation of the space of lines}\label{sec:Representation_of_Lines}

The two standard models for the oriented lines in $\BR^3$ use the tangent bundle to the unit $2$-sphere $T\BS$, where the point $(a,b)\in \BR^3\times \BR^3$, with $a\cdot a=1, a\cdot b=0$, represents the line $b+ta, t\in \BR$ (in parametrised form) or $x\wedge a=b$ (in equation form). These representations depend on the metric  and a choice of centre for the $2$-sphere. Because we are looking at affine invariants it is artificial to use the metric and we proceed as follows. 

The affine group of the real line ${\it Aff}_1$ can be realised as the matrix group
\[
\Big\{\begin{pmatrix} s&t\\0&1\end{pmatrix}:s,t\in \BR, s\ne 0\Big\}<GL(2,\BR)
\]
with the orientation preserving subgroup ${\it Aff}^+_1$consisting of those elements with $s>0$.

Consider the map 
\[
{\it Aff}_1^+\times (\BR^3\setminus \{0\}\times \BR^3) \to  \BR^3\setminus \{0\}\times \BR^3,\ (t,s).(a,b)=(sa,b+ta).
\]
This is a free action. The quotient can be identified with the set of oriented lines in $\BR^3$. The point corresponding to $(a,b)$ in this quotient, denoted $[a,b]$ represents the line $b+ta$. The tangent space to the orbit ${\it Aff}^+_1.(a,b)$ is spanned by $(a,0),(0,a)$, so the tangent space to the quotient space at $[a,b]$ can be identified with the quotient $\BR^3/sp\{a\}\oplus \BR^3/sp\{a\}$. There is a natural action of the affine group ${\it Aff}_3=GL_3\ltimes \BR^3$ on this space by
\[
(A,T)*[a,b]=[Aa,Ab+T].
\]
Writing $\BR^{3*}$ for the space dual to $\BR^3$ we have an action of $\BR\setminus\{0\}$ on $\BR^{3*}\setminus\{0\}\times \BR, s.(c,d)=(sc,sd)$, with a representative of the quotient $[c,d]$ corresponding to the plane $c(x)=d$. Again we have an action of ${\it Aff}_3$ given by
\[
(A,T)*(c,d)=[c\circ A^{-1},d+ c\circ A^{-1}(T)].
\]
Finally we have a free action of $\BR\setminus\{0\}$ on $\BR^3\setminus\{0\}$ by multiplication.

As above these representations of lines and planes depend on an arbitrary choice of origin; we shall refer to the quotient spaces above as $\BL$, $\BP$, $\BS(\BR^3)$ respectively. They can each be given a natural topology and the structure of smooth manifolds in the usual way. 

The key to much of the geometry we will uncover are incidence correspondences, namely the following sets, all of which only use the affine structure on $\BR^3$.

\begin{enumerate}
\item $I(\BR^3,\BL)=\{(p,L)\in\BR^3\times\BL:p\in L\}$

\item $I(\BP,\BL)=\{(P,L)\in\BP\times \BL:L\subset P\}$

\item $I(\BL,\BL)=\{(L,L')\in \BL\times\BL:L\cap L'\ne \emptyset\}$

\item $\tilde I(\BL,\BL)=\{(L,L',p)\in \BL\times \BL\times \BR^3:p\in L\cap L'\}$

\item $I(\BR^3,\BP,\BL)=\{(p,L,P):p\in L\subset P\}$.
\end{enumerate}

\noindent The first two sets are clearly smooth manifolds of dimension $5$, with their projections to either factor submersions. The third is smooth of dimension $7$ away from the diagonal but highly singular along it. These singularities are \lq resolved\rq\ by the fourth which is a smooth $7$-dimensional manifold as we shall see; one could equally use the corresponding space with $\BP$ replacing $\BR^3$. The fifth is smooth of dimension $6$.

\begin{rems}
{\rm 
\noindent (1) If we were working in real projective $3$-space the model above for the set of lines is replaced by the usual Klein quadric, and in the incidence sets $\BR\BP^3$ (resp. the dual space $\BR\BP^{3*}$) replaces $\BR^3$ (resp. $\BP$).

\noindent (2) It is natural to consider lines as determining loci in $\BR^3$, but it is equally natural to consideer the set in the dual space $\BP$ consisting of the planes containing the line. 

}
\end{rems}

\section{Structures on the tangent bundle to $\BL$}\label{sec:structuresTgbundle}

Given $L\in \BL$ we can consider the set $\{L'\in \BL: L\cap L'\ne \emptyset\}$. If $L'\cap L\ne \emptyset, L'\ne L$ then it easy to see that this set is smooth near $L'$ but singular at $L$. It is more convenient to work with a closed set, so we actually consider
$$
I(L)=\{L'\in \BL:L\cap L'\ne \emptyset\ {\rm or }\, L'\ {\rm parallel\ to}\ L\}.
$$

\begin{theo}\label{theo:singI(L)}
{\rm
\noindent (1) The set $I(L)$ is smooth apart from an $A_1$ singularity at $L$.

\noindent (2) The tangent cone to this singularity is a nonsingular quadric cone in $\mathbb P(T_L\BL)$ with signature $(+,+,-,-)$.

\noindent (3) If $L'\in I(L)$ and $L\ne L'$  then $I(L)\cap I(L')$ consists of the lines in the plane spanned by $L, L'$. 

\noindent (4) The set $\tilde I(\BL,\BL)$ is smooth, and the projection $\tilde I(\BL,\BL)\to I(\BL,\BL)$ is a diffeomorphism off the inverse image of the diagonal. 
}
\end{theo}

\begin{proof}
{\rm
\noindent (1) It is enough to prove this for one line $L$, e.g. $[e_3,0]$, where 
$(e_1,e_2,e_3)$ denotes the standard basis of $\mathbb R^3$. Locally we can parametrise $\BL$ as 
\[
a=(a_1,a_2,1),\ b= (b_1,b_2,0)\ {\rm representing\ the\ line\ }b+ta.
\]
Such a line meets $L=(e_3,0)$ if $e_3,\ a,\ b$ are coplanar that is $a_1b_2-a_2b_1=0$, so near $L$ we have a surface with an $A_1$ singularity at the origin. Note that this set also contains the lines parallel to $(e_3,0)$.   Parts (2), (3) are immediate. 

\noindent (4) Choosing affine charts as above, with the lines corresponding to $(a,b), (a',b')$, consider $ta+b=t'a'+b'$. Clearly $t=t', (b_j'-b_j)=t(a_j-a_j')$ and the set is parametrised by the map $\BR^7\to\BL\times\BL\times \BR^3$ taking $(t,a_1,a_2,b_1,b_2,a_1',a_2')$ to  
\[
(a_1,a_2,b_1,b_2,a_1',a_2',t(a_1-a_1')+b_1,t(a_2-a_2')+b_2,t(a_1,a_2,1)+(b_1,b_2,0)),
\]
which is clearly an embedding. The rest now follows.
}
\end{proof}
The following result is convenient for calculations later.
\begin{cor}
{\rm
\noindent With the parametrisation of $\BL$ above the tangent cone at each point is $a_1b_2-a_2b_1=0$. 
}
\end{cor}
\begin{proof}
{\rm
From above, fixing $L=(\bar a,\bar b)$ the point $(\bar a',\bar b')$ lies in the $I(L)$ if and only if $(a_1-a_1')(b_2-b_2')-(a_2-a_2')(b_1-b_1')=0$; writing $a_i'=a_i+s\a_i, b_i'=b_i+s\b_i$ one checks that the tangent cone is $\a_1\b_2-\a_2\b_1=0$.
}
\end{proof}

So at each point $L\in \BL$ the tangent cone to $I_L$ is a non-singular quadric $Q_L$ in the projectivised tangent space $\BP(T_L\BL)$ with signature $(+,+,-,-)$. This quadric has some interesting features. Viewing $T_L \BL$ as the set of infinitesimal deformations of $L$, one can keep the direction of the nearby lines constant while changing their position. This yields a natural $2$-dimensional family of nearby parallel lines, with a corresponding $2$-dimensional tangent space $W_\infty\subset T_L \BL$. On the other hand if we fix a point $x$ on the line we can consider all nearby lines through that point. The set of all oriented lines through a point is diffeomorphic to a $2$-sphere so the corresponding tangent space is another $2$-dimensional subspace $W_x\subset T_L\BL$. Intuitively we can think of $W_{\infty}$ as the case when the point $x$ has gone to infinity.  Since the only line through two different points of $L$ is $L$ itself, distinct pairs of these $2$-dimensional subspaces of $T_L \BL$ are complementary.

So we have a $1$-parameter family of $2$-dimensional subspaces of the $4$-dimensional vector space $T_L\BL$. Since the set of all lines meeting $L$ (or parallel to $L$) is the union of the set of all lines through points of $L$ (or parallel to $L$) the union of these planes give the quadric cone above. That is the local infinitesimal model of $\BL$ at each point determines a smooth projective quadric surface $Q_L$ containing a $1$-parameter family of generators, one for each point of $L$, including the one at infinity. Below we write $Q_L$ for the quadric and its equation. A complementary approach is to consider the planes $P$ that contain the line $L$; again there is clearly a $1$-parameter family, those in the pencil determined by $L$. Each one determines a $2$-parameter family of lines $L(P)$, namely the lines in that plane, and $T_LL(P)$ is a plane in $T_L\BL$ denoted $W_P$. So we obtain a second set of generators of the quadric. The only line in two distinct planes of the pencil is $L$ itself so again any distinct pairs are complementary. This is all set out in the following result where we use the parametrisation in the proof of \mbox{Theorem \ref{theo:singI(L)}}.

\begin{prop}\label{theo:Tangentspace&Q_L}
{\rm
\noindent (1) For each $x=b+ta$ we have $W_x=\sp\{(1,0,-t,0), (0,1,0,-t)\}$, and $W_{\infty}=\sp\{(0,0,1,0),(0,0,0,1)\}$.

\noindent (2) $W_x\oplus W_{\infty}=T_L\BL$, and if $x_1,\ x_2\in L, x_1\ne x_2$ then $W_{x_1}\oplus W_{x_2}=T_L\BL$. 

\noindent (3) The sets $W_x, W_{\infty}$ determine lines in the projective space $\BP(T_L\BL)$ which are one set of generators of the quadric $Q_L$, the tangent cone above.

\noindent (4) For each plane $P$ containing $L$ we have a $2$-dimensional subspace $V_P\subset T_L\BL$ and if $P_1,\ P_2$ are distinct planes then $V_{P_1}\oplus V_{P_2}=T_L\BL$. The $V_P$ determine lines in the projective space $\BP(T_L\BL)$ which are the other generators of the quadric $Q_L$. The generator corresponding to the plane $c_1x_1+c_2x_2 -(c_1a_1+c_2a_2)x_3=c_1b_1+c_2b_2$ is $sp\{(-c_2,c_1,0,0), (0,0,-c_2,c_1)\}$. 

\noindent (5) Each point of $Q_L$ determines a point $x$ of $L$ and a plane $P$ containing $L$. The tangent plane to the quadric at that point corresponds to the tangent plane to $I(L')$ at $L$ for any line $L'$ with $x\in L' \subset P$. Indeed $Q_L$ is parametrised as a ruled surface by the map $\BR\BP^1\times \BR\BP^1\to \BR\BP^3, ((\a_1,\a_2),(\b_1,\b_2))\mapsto (\a_1\b_1,\a_2\b_1,\a_1\b_2,\a_2\b_2)$ with such a point corresponding to the values $t=-\b_2/\b_1, \a_1c_1+\a_2c_2=0$ with the point at infinity when $\b_1=0$. 

\noindent (6) A line $\ell$ in $\BP(T_L\BL)$ determines a dual line $\ell^*$ in the dual projective space: the pencil of planes containing $\ell$. If $\ell$ cuts $Q_L$ at two points then $\ell^*$ cuts $Q^*_L$ in two points, so we have two planes in $\BP(T_L\BL)$. These are the tangent planes to $Q_L$ at the points of intersection of $\ell$ with $Q_L$ which intersect in $\ell$. If $\ell$ is tangent to $Q_L$ then $\ell^*$ is tangent to $Q_L^*$, if $\ell$ is a generator of $Q_L$ then $\ell^*$ is a generator of $Q_L^*$. 

\noindent (7) The fibre of the projection $I(\BR^3,\BL)\to \BR^3$ (resp. $I(\BP,\BL)\to \BP$) over any point $x$ (resp. plane $L$) is the set of lines through $L$ (lines in the plane $P$).
}
\end{prop}

\begin{proof}
{\rm
\noindent (1)-(3) In the chart in the proof of Theorem \ref{theo:singI(L)} a line 
$$
(a',b')=((a'_1,a'_2,1),(b'_1,b'_2,0))
$$ 
passes through $b+ta$ if $b-b'=t(a'-a)$, that is $(a',b')=(a',b+t(a-a'))$. At $(a,b)$ the corresponding tangent plane is spanned by $\{(1,0,-t,0), (0,1,0,-t)\}$  which determines a generator of the quadric $\a_1\b_2=\a_2\b_1$. The expression for $W_{\infty}$ is clear. 

\noindent (4) Now consider a plane $c(x)=d$ this contains $(a',b')$ if and only if $c(a')=0,c(b')=d$. Writing $c=c_1e_1^*+c_2e_2^*+c_3e_3^*$, ($e_i^*(e_j)=\d_{ij}$), a general such plane satisfies $c_1a_1'+c_2a_2'+c_3=0, c_1b_1'+c_2b_2'=d$. Fixing $[c,d]$ and assuming that $c_1\ne 0$ we have a parametrisation $(-(c_2a_2'+c_3)/c_1,a_2',(d-b_2'c_2)/c_1)$ and differentiating with respect to $a_2', b_2'$ and setting $(a,b)=(a',b')$ we obtain the indicated tangent space determining the other rulings on the quadric.

\noindent (5) With the above established we can deduce the first part without calculations. If $L'\in I(L)$ but $L'\ne L$ consider $I(L')$. This passes through $L$ and is smooth there so $T_LI(L')$ is a hyperplane in $T_L\BL$. Clearly $L$ and $L'$ determine a point of intersection $x(L,L')$ and a plane containing both, say $P(L,L')$. Hence they determine a point on the projective quadric which is the intersection of the two corresponding generators. Indeed $T_LI(L)$ is the tangent hyperplane determined by the tangent plane to the quadric at this point. This is because the lines through $x(L,L')$ are a subset of $I(L')$ as are the lines in the plane $P(L,L')$.  The second is straightforward.

\noindent (6) These are just general facts about quadrics with signature $(+,+.-,-)$. 

\noindent (7) This is immediate.
}
\end{proof}

\begin{rems}
{\rm
\noindent (1) The quadric arises from a natural non-degenerate, indefinite metric on $\BL$ discussed for example in \cite{GK} and \cite{shepherd},  which depends on the Euclidean structure. A quadric with signature $(+,+,-,-)$ (resp. $(+,+,+,-)$) has the property that a line meets it in $2$ points if and only the dual line meets the dual quadric in $2$ points (resp. the dual line does not meet the dual quadric). As usual the quadric allows one to identify the projective space with its dual via polars. 

\noindent (2) Given a line $L\in \BL$ we can consider the subgroup of ${\it Aff}_3$ fixing $L$, that is $G_a\ltimes \BR_a$, where $G_a\subset GL_3$ is the stabiliser of the subspace spanned by $a$, and $\BR_a$ is the set of translations in the direction of $a $. One can check that the quadric and the distinguished line are the only sets invariant under the action of this group. 

\noindent (3) Note that although there is a $3$-parameter family of points $L'\in I(L)$ we only obtain a $2$-parameter family of hyperplanes $T_LI(L')$. This is because given a point $x\in L$ and a plane $L\subset P$ there is a pencil of lines $L''$ through $x(L,L')$ contained $P(L',L)$ each determining the same point of the quadric. Indeed corresponding to each point of $Q_L$ there is a plane pencil of lines in $\BL$ passing through a fixed point and lying in a fixed plane, and the tangent plane to $Q_{L'}$ at that point is the tangent space to $I(L')$ for any $L'$ in that pencil. 

\noindent (4) Clearly a curve of lines $L:\BR\to \BL$ (that is a ruled surface) is developable if and only if its tangent vector $L'(s)$ lies in the quadric cone $Q_{L(s)}\subset T_{L(s)}\BL$ for each $s$. 

\noindent (5) In the real projective case the above remains true with the obvouis interpretations, but there is no distinguished generator corresponding to parallel lines. 
}
\end{rems}

\section{The complex case}\label{sec:complexcase}

The results above are valid replacing $\BR$ by $\BC$, with the usual advantages, but we are primarily interested in the real case.  Of course some lines do not meet real projective quadrics, and this makes the usual geometric interpretation of polars problematic. However any line in the complex projective space $\BP(T_{L}\BL)_{\BC}$ will meet the complexified  quadric $Q_{\BC L}$ in two (possibly coincident) points. Suppose given a line $\ell$ which meets $Q_{\BC L}$ in just two points $x, x'$. These generally determine two complex points $p, p'$ on $L$; we shall presume $\ell$ does not meet the distinguished generator of lines parallel to $L$. 

\begin{prop}\label{prop:middlepointsComplex}
{\rm
When the line $\ell$ is real the midpoint $\frac{1}{2}(p+p')$ is also real. 
}
\end{prop}
\begin{proof}
{\rm
\noindent If $x,x'$ are the points of intersection with the quadric then $Q_L(x)=Q_L(x')=0$; since $Q_L$ is real taking complex conjugates $Q_L(\bar x)=Q_L(\bar x')=0$, so since $\ell$ only meets $Q_L$ in two points either $x, x'$ can be chosen to be real or $x=\lambda \bar x'$ for some non-zero complex number $\lambda$. Now $x=b+ta$ is on the generator $W_t=W_x$ (resp. $x'$ is on the generator $W_{t'}$) above if and only if for some $\a, \b, \a', \b'$
\[
\a(1,0,-t,0)+\b(0,1,0,-t)=(x_0,x_1,x_2,x_3),{\rm and}
\]
\[
\a'(1,0,-t',0)+\b'(0,1,0,-t')=\lambda (\bar x_0,\bar x_1,\bar x_2,\bar x_3).
\]
Taking conjugates of the first equality
\[
\bar\a(1,0,-\bar t,0)+\bar\b(0,1,0,-\bar t)=(\bar x_0,\bar x_1,\bar x_2,\bar x_3),
\]
writing $\mu=\lambda^{-1}$, and subtracting, we deduced that
\[
\a'\mu-\bar \a=\b'\mu-\bar \b=\a'\mu t'-\bar \a\bar t=\b'\mu t-\bar \b\bar t=0.
\]
In particular $\bar \a(t'-\bar t)=\bar \b(t'-\bar t)=0$ so either $t'-\bar t$ or $\a'=\b'=0$. The former implies that $\frac{1}{2}(t+t')$ is real, the second gives a contradiction.
}
\end{proof}

\section{The geometry of congruences}\label{sec:geomCong}

A congruence of lines is a smooth immersed surface $Z\subset \BL$; a generic map $Z\to \BL$ will be an immersion with isolated transverse double points.  The internal structure carried by $\BL$ has significant geometric consequences for $Z$. Each point $z\in Z$ determines a line $L(z)\in\BL$ and a tangent plane $T_zZ\subset T_{L(z)}\BL$, which in turn gives a line $\ell_z$ in the projectivised tangent space $\BP(T_{L(z)}\BL)$ containing the quadric $Q_{L(z)}$. The set $\{(x,L)\in \BR^3\times \BL:x\in L, L\in Z\}$ (resp. $\{(P,L)\in \BP\times \BL:L\subset P, L\in Z\}$) is denoted $\BR^3_Z$ (resp. $\BP_Z$).

\begin{theo}\label{theo:genrators_Q_Cong}
{\rm
\noindent (1) The line $\ell_z$ meets $Q_{L(z)}$ in $0$ or $2$ points, or is tangent to it, or is a generator of it.  The lines tangent to $Q_{L(z)}$ form a real algebraic subset of the real $4$-dimensional Klein quadric $\BK$ of codimension $1$, indeed its intersection with another quadric in $\BR\BP^5$, so generically there will be a curve of points $z\in Z$ with $\ell_z$ tangent to $Q_{L(z)}$, separating points where $\ell_z$ has $0$ or $2$ intersections. The set of generators of $Q_{L(z)}$ yield two smooth disjoint curves in $\BK$, so for a generic congruence no $\ell_z$ for $z\in Z$ is a generator, though this occurs in $1$-parameter families of congruences.

\noindent (2) At each point of intersection of $\ell_z$ with $Q_{L(z)}$ there are two generators; one corresponds to a point on $L(z)$ and the other to a plane containing $L(z)$. Each corresponds to a plane pencil of lines in $\BL$ determined by the corresponding tangent space to $Q_{L(z)}$.

\noindent (3) Generically there is a smooth curve of points in $Z$ where $\ell_z$ meets $W_{\infty}$.  

\noindent (4) The mid-point of the two points on $L(z)$ determined by the intersection of $\ell_z$ and $Q_{L(z)}$  is real even if $\ell_z$ meets $Q_{\BC L}$ in non-real points.

\noindent (5) The projection $\pi:\BR^3_Z\to \BR^3$ (resp. $\pi:\BP_Z\to \BP$) will have a critical point at $(x,L(z))$ (resp. $(P,L(z))$) if and only $x$ (resp. $P$) is one of the points (resp. planes) determined by the intersection of $\ell_z$ with $Q_{L(z)}$.

\noindent (6) If $\ell_z$ cuts the quadric in $2$ points the tangent planes to $Q_{L(z)}$ at those points meet in another line denoted $\ell_z^*$.
}
\end{theo}

\begin{proof}
{\rm
\noindent (2), (4), (6) are clear, and we prove (1) and (3) below. Part (4) follows from the previous proposition. For (5) the first (resp. second) projection is singular at $(x,L)$ (resp. $(P,L)$) if and only if the set of lines through $x$ (contained in $L$) is tangent to $Z$.
}
\end{proof}

\begin{defn}\label{def:hyp_par_el}
{\rm
\noindent (1) A point $z\in Z$ is said to be respectively {\it hyperbolic, elliptic, parabolic}, if $\ell_z$ meets the quadric $Q_{L(z)}$ in $2, 0$ or $1$ points. At a non-elliptic point the directions on $Z$ determined by the intersections are called {\it torsal directions}. 

\noindent (2) If $\ell_z$ meets $Q_{L(z)}$ at a point corresponding to a pair $(x,P)\in \BR^3\times \BP$ with $x\in L\subset P$ then $x$ is a {\it focal point} of $Z$ at $z$ and $P$ a {\it focal plane}.

\noindent (3) The locus of midpoints of focal points (from above defined even at elliptic points) is called the {\it middle surface} of $Z$.

\noindent (4) Given a hyperbolic (resp. parabolic) point $L(z)$ there are $2$ (resp. $1$) point(s) of intersection with $Q_{L(z)}$ each determining a line in $T_zZ$, the  torsal direction for $Z$ at $L(z)$. (These determine the {\it torsal BDE} of the congruence, defined on the region of non-elliptic points of $Z$.) 

\noindent (5) The map $U\to \BS(\BR^3), (u,v)\mapsto a(u,v)$ is called the {\it direction map} of the congruence; a singular point of the map is called a {\it stall point}.
}
\end{defn}

Suppose $Z$ is parametrised locally as $(\bar a,\bar b)$. Away from stall points we may suppose $a_1=u, a_2=v$, since the direction map is then a local diffeomorphism. On the other hand we shall show below that any generic congruence is locally parametrised as $(a_1,a_2,u,v)$. These parametrisations simplify computations; we generally use the first. 

\begin{prop}\label{prop:flatdirections}
{\rm
\noindent (1) The directions in $T_zZ$ determined by the intersection with the quadric, equivalently the torsal directions, are given by the binary differential equation (BDE)
\begin{equation}\label{eq:BDEfocaldirec}
(a_{1u}b_{2u}-a_{2u}b_{1u})du^2+(a_{1u}b_{2v}+a_{1v}b_{2u}-a_{2u}b_{1v}-a_{2v}b_{1u})dudv+(a_{1v}b_{2v}-a_{2v}b_{1v})dv^2=0.
\end{equation}
In the case of the parametrisations $(u,v,b_1,b_2)$ (resp. $(a_1,a_2,u,v)$), this becomes respectively
\[
b_{2u}du^2+(b_{2v}-b_{1u})dudv-b_{1v}dv^2=0,\quad  a_{2u}du^2+(a_{2v}-a_{1u})dudv-a_{1v}dv^2=0.
\]
\noindent (2) Any curve on $Z$ through $z$ in these directions determines a ruled surface which has a generator at $z$ whose parameter of distribution/pitch vanishes. That is the integral curves of this BDE are the developable surfaces  (or to use a more antiquated term torsal surfaces) on $Z$. 

\noindent (3) The two focal points $b+ta$ and focal planes as in Proposition \ref{theo:Tangentspace&Q_L}(4) are given respectively by the vanishing of the determinants:
\[
\begin{vmatrix} 1&0&-t&0\\0&1&0&-t\\a_{1u}&a_{2u}&b_{1u}&b_{2u}\\a_{1v}&a_{2v}&b_{1v}&b_{2v}\end{vmatrix},\quad  \begin{vmatrix} -c_2&c_1&0&0\\0&0&-c_2&c_1\\a_{1u}&a_{2u}&b_{1u}&b_{2u}\\a_{1v}&a_{2v}&b_{1v}&b_{2v}\end{vmatrix}.
\]
With the simplified parametrisations above these become respectively:
\[
t^2+t(b_{1u}+b_{2v})+(b_{1u}b_{2v}-b_{1v}b_{2u})=0;\ c_1^2b_{1v}+c_1c_2(b_{2v}-b_{1u})-c_2^2b_{2u}=0
\]
\[
t^2(a_{1u}a_{2v}-a_{1v}a_{2u})+t(a_{1u}+a_{2v})+1=0;\ c_1^2 a_{1v}+c_1c_2(a_{2v}-a_{1u})-c_2^2a_{2u}=0. 
\]

\noindent (4) One of the focal points has gone to infinity if and only if we are at a stall point; both have if in addition we are at a parabolic point.

\noindent (5) The middle point surface is parametrised as $(b_1,b_2,0)-\frac{1}{2}(b_{1u}+b_{2v})(u,v,1)$, 
(resp. $(u,v,0)-\frac{1}{2}\frac{a_{1u}+a_{2v}}{a_{1u}a_{2v}-a_{1v}a_{2u}}(a_1,a_2,1)$).
}
\end{prop}
 
\begin{proof}
{\rm
\noindent (1), (3). An element of the tangent space is given by $u'\partial(\bar a,\bar b)/\partial u+v'\partial(\bar a,\bar b)/\partial v$ determining a point in $\BP(T_L\BL)$ which lies on the quadric when the equality above holds. This point lies on the generator $W_x$ corresponding to $x=b+ta$ if and only if these partial derivatives, $(1,0,-t,0), (0,1,0,-t)$ are dependent. Similarly for the $V_P$ generators. For $W_{\infty}$ the relevant $4\times 4$ determinant is the condition that $(a_1,a_2)$ is singular. The rest is straightforward.

}
\end{proof} 
\begin{rems}
{\rm
\noindent (1) The definitions in Definition \ref{def:hyp_par_el} and geometric configurations are affine invariants of $Z$. Recall that we have an action of ${\it Aff}_3$ on $\BL$, so if $\phi\in {\it Aff}_3$ there is an associated congruence $\phi^*Z$. The claim is that the focal points, planes, middle surface of $Z$ are taken to those of $\phi^*Z$ by $\phi$.

\noindent (2) There are a range of ways of studying congruences; for example by analogy with the classical study of the extrinsic geometry of surfaces (\cite{weatherburn}), or using frames (\cite{shepherd}). The classical literature does not distinguish affine and Euclidean invariants.  

\noindent (3) Note that the focal points and planes are paired. The quadric is parametrised as the image of a map $\BR\BP^1\times \BR\BP^1\to \BR\BP^3$ as above in Proposition \ref{theo:Tangentspace&Q_L}(5) and a point determined by the intersection of $\ell_z$ with the quadric determines a point and a plane.

}
\end{rems}

We shall uncover a range of different phenomena, and to relate them one can write down the conditions they impose on the Taylor series coefficients of the $a_i, b_i$. The key to genericity questions are transversality results; that presented here is the analogue of one for surfaces in $\BR^3$ given in \cite{reflexions}, and a result concerning incidence sets is established in Section 7. Because we will want to use the first transversality result when considering the Euclidean geometry of congruences and complexes we use the usual Euclidean model for $\BL$ as the tangent bundle to the unit sphere $\BS$ in $3$-space.  The bundle $T\BS$ is covered by $3$ open sets, $U_j$ consisting of those lines not orthogonal to the axes $e_j, 1\le j\le 3$. If ${\it Euc}$ is the Euclidean group there is a smooth map $\phi_j:U_j\to {\it Euc}$ which associates to a line $L=(a,b)$ the translation taking $x$ to $0$ and then the rotation in the line orthogonal to the trace of $L$ and the $e_j$-axis which takes that image to the latter axis. So for each $L\in U_j$ we have a smooth map $\phi_j(L)^*:U_j\to T\BS, L\mapsto \phi_j(L)(L')$ taking $L$ to $[e_j,0]$; indeed the group $Euc$ clearly acts smoothly on $T\BS$, and $\phi_j^*$ is a diffeomorphism onto its image. 

For any manifolds $X,Y$ we write $J_0^k(X,Y)$ for the set of $k$-jets of mappings $X\to Y$ with no constant terms.  Given an immersion $i: Z\hookrightarrow T\BS$ if we set $Z_j=Z\cap U_j$ then we obtain a map $\a_j:Z_j\to J_0^k(Z_j,\BR^4)$ given by $L\mapsto j^k(\bar a,\bar b)$ where we consider the congruence $\phi_j(L)^*(Z)$ and its expression at $\phi_j(L)*(L)=[e_j,0]$ as $(\bar a,\bar b)$ in the usual way, replacing $e_3$ by $e_1, e_2$ for $Z_1, Z_2$ of course. The idea is to identify subsets of the jet space $J_0^k(Z_j,\BR^4)$ corresponding to geometric phenomena and ask that the maps above are transverse to them. If the subsets are to be of interest they need to be invariant under arbitrary changes of co-ordinates in the source and the action of the Euclidean group, or the affine group for this paper on the target. More precisely if we consider all Euclidean (or affine) transformations $G_j$ that fix the line $[e_j,0]$, the subsets of $J_0^k(2,4)$ should be invariant under the corresponding action of $G_j$, and the action of the group of right equivalences ${\cal R}$. In the Euclidean case the subgroup preserving the line is the semi-direct product of rotations about and translations along the line. In the affine case we have the semi-direct product of $GL_2\ltimes \BR$. We call such subsets of $J_0^k(Z_j,\BR^4)$ {\it Euclidean (affine) invariant}; note that for the Euclidean or affine invariant submanifolds transversality is independent of the choice of chart $U_j$ above. Once we have fixed a chart on $Z$ they are subssets of $J^k_0(2,4)$, that is the set of $4$ polynomials of degree $\le k$ in $2$ variables with no constants.

\begin{prop}\label{prop:transv}
{\rm 
\noindent (1) Given a non-elliptic point $L_0=[e_3,0]$ in a congruence $(u,v,b_1,b_2)$ by a rotation about $L_0$ and a translation along it we may suppose that $b_{1u}=b_{2u}=0$ at $u=v=0$. 

\noindent (2) Given a hyperbolic point $L_0=[e_3,0]$ we can further suppose  that either $b_{1u}=b_{2u}=b_{1v}=0, b_{2v}=1$. 

\noindent (3) Let $Z$ be a smooth compact manifold of dimension $1,2$ or $3$. Given a countable set of Euclidean invariant submanifolds $W_j\subset J_0^k(Z,\BR^2), j\in J$, then for a generic immersion $i:Z\hookrightarrow T\BS$, with image contained in a bounded region of $T\BS$, that is for which $b$ is bounded, the map above $Z\to J_0^k(Z,\BR^4)$ is transverse to the $W_j$. 

\noindent (4) For a generic congruence $Z$ at each point $L\in Z$ we can choose a local parametrisation of $Z$ of the form $(a_1(u,v),a_2(u,v),u,v)$.
}
\end{prop}

\begin{proof}
{\rm
\noindent (1) Since we are at a non-elliptic point we have a focal point, a focal plane and a torsal direction. By a rotation about the line $L_0$ and translation along it we may assume that $0$ is a focal point with $x_2=0$ the corresponding focal plane. From this it follows that at $u=v=0$ in Proposition 3.3 (5) we have $\a_2=\b_2=0$, so the corresponding point of the quadric is $(1,0,0,0)=u'(1,0,b_{1u},b_{2u})+v'(0,1,b_{1v},b_{2v})$. So $v'=0$, our torsal direction is $(1,0)$ and $b_{1u}=b_{2u}=0$. This simplifies matters - for example the point is parabolic if in addition $b_{2v}=0$ there.

\noindent (2) If we are at a hyperbolic point then we have a second focal point and plane, which we may suppose, after an affine transformation, are given by $-e_3$ and $x_1=0$. We deduce that in addition $b_{2v}=1$ and $b_{1v}=0$.

\noindent (3) Let $P^k$ be the set of polynomial maps $\BR^6\to \BR^6$ of degree $\le k$. Suppose that $Z\subset T\BS\subset \BS\times B_r$ where $B_r$ is the ball of radius $r$ centred on the origin. Choose a bounded open subset $U_1$ of $\BR^3$ containing the unit ball $B_1$, and another $U_2$ containing $B_r$. So $Z\subset  \BS\times B_r\subset U_1\times U_2$. Now choose a neighbourhood $V$ of the identity map $Id\in P^k$ with the property that if $P\in V$ then the restriction of $P$ to $B_1\times B_2$ is a diffeomorphism onto its target and if $(\a,\b)\in B_1\times B_2$ has $||\a||>\frac{1}{2}$ then writing the first three components of $P$ as $P_1$ we have $P_1(\a,\b)\ne 0$. There is a natural submersion
\[
\begin{array}{rcl}
\Phi:\BR^3\setminus\{0\}\times \BR^3&\to& T\BS\\ 
(y_1,y_2)&\mapsto& (\frac{1}{||y_1||}y_1,y_2-\frac{y_1\cdot y_2}{y_1\cdot y_1}y_1)=(a,b)
\end{array}
\]
with the oriented line $y_2+ty_1$ corresponding to $(a,b)\in T\BS$. We can also shrink $V$ if needs be to ensure that for $P\in V$ the map $Z\to T\BS, z\mapsto \Phi(P(z))$ is still an immersion for all $\Phi\in V$. Now consider the map 
\[
Z\times V \to J_0^k(Z,\BR^4),
\]
given by applying the map $\a_j$ to the immersed submanifold $\Phi(P(Z))$; strictly speaking this map is only defined on an open subset of $Z$ as usual, but this causes no problems. The claim is that this map is a submersion. We first prove this is the case at points $(z,Id)\in Z\times V$. Consider the map $H:Z\times V\to \BR^3\setminus\{0\}\times \BR^3, (z,P)\mapsto P(z)$ and its jet extension $j_1^kH:Z\times V\to J^k(Z,\BR^3\times \BR^3)$. We claim that this is a submersion. The polynomial mappings $(x,y)\mapsto (x+tp(x,y)e_j,y)$ and $(x,y)\mapsto (x,y+tp(x,y)e_j)$ where $t\in \BR$, $1\le j\le 3$ and $p$ a polynomial clearly lie in $V$ for $t$ small. Differentiating with respect to $t$ taking charts we obtain $j^k(p(x(u,v),y(u,v))e_i,0), j^k(0,p(x(u,v),y(u,v))e_i)$ and since $(x,y):Z\to \BR^3\times \BR^3$ is an immersion the result follows. On the other hand $\Phi$ is a submersion, so the jet-extension 
\[
j_1^k(\Phi\circ H):Z\times V\to J^k(Z,T\BS)
\] 
is also a submersion. Now consider the general case and a point $(z,P)\in Z\times V$. We know that $\Phi(P(Z))$ is a compact immersed submanifold of $T\BS$, just as $Z$ was. Now $P$ has invertible linear part at $(0,0)\in \BR^3\times \BR^3$ so for some polynomial map $Q\in P^k$ we have $Q\circ P=Id$ modulo terms of degree $>k$. Now apply the same argument as above but with $(x+tp(x,y)e_j,y)$ and $(x,y+tp(x,y)e_j)$ first composed with $Q$ and the composite truncated to degree $k$. We now use Thom's basic transversality lemma \cite{GG}, p.53, to deduce that for all polynomials $P\in V$ off a set of measure zero our map applied to the manifold $P(Z)$ is transverse to each $W_i$, and of course a countable union of sets of measure zero has measure zero. 

\noindent (4) Suppose we parametrise our surface locally by the direction of the line and its intersection with the plane $x_3=c$. This replaces the functions $b_1,b_2$ by $b_1+ca_1, b_2+ca_2$. So we only have a problem when $(b_1+ca_1,b_2+ca_2):\BR^2,0\to \BR^2$ is not a local diffeomorphism for all $c$; in other words every point on the line is focal. This means that the first determinant in Proposition \ref{prop:flatdirections}(3) vanishes for all $t$. One easily checks that this determines an affine invariant algebraic subset of $J_0^1(2,4)$ of codimension $3$, so by (3) for a generic congruence this does not happen at any point. 
}
\end{proof}

\begin{rems}
{\rm
\noindent (1) The compactness condition is not essential, but simplifies the proof.

\noindent (2) We often put the congruence locally in a pre-normal form, for example $(u,v,b_1,b_2)$, and then geometric conditions emerge as polynomials in the Taylor series of the $b_i$. The same arguments then apply: provided the subset of $J^k_0(2,2)$ has codimension $\ge 3$, that is can be stratified into countably many such manifolds, then generically it can be avoided.

\noindent (3) As an alternative approach we can show that submanifolds of $J^k_0(2,4)$ which are Euclidean or affine invariant give rise to submanifolds of $J^k(Z,\BL)$ of the same codimension and apply the usual Thom tranversality lemma.  
}
\end{rems}


\section{Focal sets}\label{sec:Focalsets}

The notion of point-focal set is classical but plane-focal sets seem less well studied; they are linked to the incidence sets mentions above.

\begin{defns}
{\rm
\noindent (1) Given a submanifold $Z\subset \BL$ we can consider the {\it incidence sets} $\BR^3_Z=\{(x,L)\in \BR^3\times Z:x\in L\}, \BP_Z=\{(P,L)\in \BP\times Z: L\subset P\}$. 

\noindent (2) The projection $\pi:\BR^3_Z\to \BR^3$ (resp. $\pi:\BP_Z\to\BP$) is called the point (resp. plane) exponential map. 
  
}
\end{defns}
 
\begin{theo}\label{theo:exp_maps}
{\rm
\noindent (1) For generic $Z$ the incident sets $\BR^3_Z, \BP_Z$ are smooth of dimension $\dim Z+1$ and their projections respectively to $\BR^3, \BP$ exhibit the local singularities of stable mappings. In particular for congruences $\dim Z=2$ they are locally diffeomorphisms, folds, cusps or swallowtails; see \cite{GG}. 

\noindent (2) When $\dim Z=2$ the set of critical points of the point exponential map (resp. plane exponential map) consists of the focal points (resp. focal planes) of $Z$.

\noindent (3) The kernel of the exponential map $Z\times \BR\to \BR^3$ at $(u,v,t)$ giving a focal point is the vector $(u',v',0)$ where $(u',v')$ is a torsal direction at that point. The tangent space to the focal surface is spanned by the line $L(z)$ and the image of the other torsal direction. 
}
\end{theo}

\begin{proof}
{\rm
\noindent (1) The sets $\BR^3_{\BL}$ and $\BP_{\BL}$ are clearly smooth with the natural projections to $\BL$ submersions. The incidence sets are the inverse images of $Z$ and hence smooth.  To understand the nature of the singularities it is easiest to write our lines in direction/directrix form, that is as $b(u,v)+ta(u,v), (u,v)\in U\subset \BR^k, 1\le k\le 3, b:U\to \BR^3, a:U\to \BR^3\setminus\{0\}$. The point-exponential map is given by
\[
E:U\times \BR\to \BR^3, (u,v,t)\mapsto b(u,v)+ta(u,v).
\]
Given maps $b:U\to \BR^3, a:U\to \BR^3\setminus\{0\}$ consider the map 
\[
\a:U\times J^k(U,\BR^3)\times J^k(U,\BR^3\setminus\{0\}) \to J^k(U,\BR^3),
\]
\[
(u,A,B)\mapsto j^k(b(u,v)+B(u,v)+t(a(u,v)+A(u,v)).
\] 
We can assume $a+A\ne 0$ for $A$ small. It is easy to see that the map $\a$ is transverse to any ${\cal A}$-invariant submanifold of $J^k(U,\BR^3)$, and by Thom's transversality lemma, for a dense set of maps $A, B$ the set of lines given by $(b+B,a+A)$ yields an exponential map which is transverse to a finite number of such manifolds. Given our dimension ranges the result follows. 

In the second case recall that a plane $[c,d]$ contains a line $[a(u,v),b(u,v)]$ if and only if $c\cdot a(u,v)=0, b(u,v)\cdot c=d$. Working locally we can change co-ordinates so that the plane and line are respectively $[e_1,0],[e_3,0]$ with $c$ specified as $e_1^*+c_2e_2^*+c_3e_3^*$. Moreover we can ensure $b(u,v)=(b_1(u,v),b_2(u,v),0), a(u,v)=(a_1(u,v),a_2(u,v),1)$ for $u,v$ near $p$. The equations for inclusion then become 
\[
a_1(u,v)+c_2a_2(u,v)+c_3=b_1(u,v)+c_2b_2(u,v)-d=0.
\]
So the projection map is given by
\[
u\mapsto [(1,c_2,-a_1(u,v)-c_2a_2(u,v)),b_1(u)+c_2b_2(u,v))].
\]
Given that $c_2$ plays the same role here as $t$ did above, the result follows by the same arguments.

\noindent (2) Consider the case of focal points; the projection $\pi:\BR^3_Z\to \BR^3$ will have singularity at $(x,L(z))$ if and only if $x\in L(z)$ and the set of lines through $x$ is tangent to $Z$. This means that in $\BP(T_{L(z)}\BL)$ the generator corresponding to $x$ meets $\ell_z$, and the result follows. The focal plane case is similar.   

\noindent (3) The vector $(\a,\b,0)$ is in the kernel of the exponential map at $(u,v,t)$ if and only if $(b_{1u}+t)\a+b_{1v}\b=b_{2u}\a+(b_2v+t)\b=0$. Now eliminate $t$.
}
\end{proof}

\begin{rems}
{\rm
\noindent It is possible to improve (1) in Theorem \ref{theo:exp_maps} above to deal with multilocal singularities too, but one needs to be careful in the first case since the domain is non-compact. 
}
\end{rems} 


\section{Second transversality result}\label{sec:transversality}

To study the focal sets we need a slightly different transversality result to that usually invoked in generic geometry which involves the notion of {\it contact-equivalence} which we now recall. Let ${\cal E}_s$ be the ring of germs of smooth functions $\BR^s,0 \to \BR$ and ${\m}_s$ its maximal ideal, the subset that vanish at the origin. Denote by ${\cal E}(s,t)$ the $t$-tuples of elements in ${\cal E}_s$.  The set of germs of diffeomorphisms $\BR^s\times \BR^t,0\to \BR^s\times \BR^t,0$ which can be written in the form $H(x,y)=(h(x),H_1(x,y)),$ with $h$ a diffeomorphism $\BR^s,0\to \BR^s,0$ and $H_1(x,0)=0$ for $x$ near $0$ is the contact group ${\cal K}$; see \cite{mather3}. The group $\cal K$ acts on ${\mathcal M}_s.{\cal E}(s,t)$ as follows: $g=H.f$ if and only if $(x,g(x))=H(h^{-1}(x),f(h^{-1}(x)))$.

Now suppose $X, Y$ are smooth manifolds and $\Gamma$ is a smooth submanifold of $X\times Y$ with the natural projections $\pi_X:\Gamma\to X, \pi_Y:\Gamma\to Y$ submersions.  Clearly for each $(x,y)\in \Gamma$ the latter is given locally as the inverse image of the regular value $0$ of some smooth map $F:X\times Y, (x,y)\to \BR^k,0$; we denote the restriction of $F$ to $:X\times\{y\}$ by $F_y:X\to \BR$. For $x\in X$ (resp. $y\in Y$) we denote by $\Gamma_x$ (resp. $\Gamma_y$) the submanifold $\pi_X^{-1}(x)$ (resp, $\pi_Y^{-1}(y)$). Given a submanifold $Z\subset X$ we can consider the projection $\pi_Y:\pi_X^{-1}(Z) \to Y$. The following is a  generalisation of a result in \cite{bruce}, which dealt with the case $k=1$.

\begin{prop}\label{prop:sectransv}
{\rm
\noindent (1) The contact between  $Z$ and $\Gamma_y$ at $x$ is determined by the ${\cal K}$-type of the restriction of the map $F_y$ to $Z\times \{y\}$. This ${\cal K}$-type is independent of the choice of $F$. The same result holds true for multi-germs at $(x_i,y)\in \Gamma, 1\le i\le r$.

\noindent (2) Let $\cal S$ be a smooth contact invariant stratification of the mult-jet space $_rJ^k(Z,\BR^k)$. For a residual set of embeddings $Z\hookrightarrow X$ the multi-jet extension
\[
_rj^k_1F:Z^{(r)}\times Y\to _rJ^k(Z,\BR^k)
\]
is multi-transverse to $\cal S$. The map $F$ cannot generally be chosen uniformly, but the above definition makes sense independent of the choices involved.
}
\end{prop}  

The proof is the same as given in \cite{bruce}; the key is that the mappings $F_y:X\to \BR^k$ are germs of submersions. We will apply this result in the cases when: $X=\BR^3$ (resp. $\BP$), $Y=\BL$ and $\Gamma=\{(x,L):x\in L\}$ (resp. $\{(P,L):L\subset P\}$), and in two further cases we explain below.


\section{Model submanifolds and contact with congruences}\label{sec:ContactModels}

Let $M$ and $N$ be two submanifolds of $\BR^n$ with $p\in M\cap N$. If we parametrise $M$ locally by an immersion $\phi:\BR^s,0\to \BR^n,p$, and $N$ near $p$ is the zero set of a submersion $\psi :\BR^n,p\to \BR^t,0$, the ${\cal K}$-type of the {\it contact map}, that is the composite $\psi\circ \phi: \BR^s,0\to \BR^t,0$, does not depend on any choices made in this construction. The {\it contact between $M$ and $N$ at $p$} is the $\cal K$-singularity type of this germ; see \cite{Montaldi}.

In studying surfaces in $3$-space we tease out their geometry by \lq comparing them\rq\ with model submanifolds. Our model of \lq roundness\rq\ might be a sphere, of \lq flatness\rq\ a line or plane. We measure roundness or flatness of a surface by considering its contact with these models.  In our context a good model submanifold is a ruled surface, congruence or complex ($3$-dimensional submanifold of $\BL$) whose symmetry group in ${\it Aff}_3$ is large.  In each case we have a submanifold of the space of lines, the analogue of the sphere, and a manifold which parameterises them, for spheres the set of centres and radii.

\begin{exs}{\rm \bf(Model submanifolds)}
	
{\rm
\noindent (1) {\it Lines through a point.} There is a $3$-parameter family of lines parametrised by points of $\BR^3$; the set of oriented lines through a point is diffeomorphic to $\BS$. 

\noindent (2) {\it Lines in a plane.} A $3$-parameter family of planes; the set of oriented lines in a plane is diffeomorphic to a cylinder, the tangent bundle to the circle.  

\noindent (3) {\it Lines in a given direction.} A family parametrised by $\BR\BP^2$, and for each point of $\BR\BP^2$ we have a $2$-parameter family of lines. 

\noindent (4) {\it Lines in parallel planes.} A $2$-parameter family of directions, each giving a set of oriented lines orthogonal to that direction, which is diffeomorphic to $\BS^1\times \BR^2$. With a Euclidean metric these are lines orthogonal to fixed direction. 

\noindent (5) {\it Lines in a plane through a point.} A family parametrised by a $5$-dimensional space $I(\BR^3,\BP)$, with each point giving a $1$ parameter family of lines. 

\noindent (6) {\it Lines meeting a line.} As above this is a hypersurface in $\BL$ with an $A_1$ singularity. This is clearly a $4$-parameter family, parametrised by points of $\BL$. Away from cone points we can consider contact in the usual way. At the cone point we do not have a smooth manifold but given a line $L\in \BL$ locally we can choose a mapping $f:\BL,L\to \BR$ with an $A_1$ singularity and $f^{-1}(0)$ the lines through $L$, and any two such mappings with $A_1$ singularities are ${\cal K}$-equivalent. So given a germ of a manifold $Z$ in $\BL$ through $L$ we can consider $f|Z:Z,L\to \BR$ which is well defined up to ${\cal K}$-equivalence. 
}
\end{exs}

\subsection{Lines through a point}

\begin{theo}\label{theo:linesthroughapoint}
{\rm
\noindent (1) The set of lines through a point $p$ is tangent to the congruence $Z$ at $L$ if and only if $p$ is a focal point of the congruence at $L$.

\noindent (2) For a generic congruence the contact at $L$ is of type $A_k$, in other words is ${\cal K}$-equivalent to $(u,v^{k+1})$, for $0\le k\le 3$. These singularities are ${\cal K}$-versally unfolded within the family of lines through a point. 

\noindent (3) Taking $p=0$ and working at $[e_3,0]$, parametrising $Z$ as $(u,v,b_1,b_2)$ the condition for singular contact is $b_{1u}b_{2v}-b_{1v}b_{2u}=0$. Supposing $b_{1u}=b_{2u}=0$ at $(0,0)$ a non-elliptic point, the condition for $A_{\ge 2}$ contact, corresponding to the cuspidal edge on the focal set, is $b_{1v}b_{2uu}-b_{2v}b_{1uu}=0$.  At a hyperbolic point we may suppose further that $b_{1v}=0$ so this become $b_{1uu}=0$ in this case. 

\noindent (4) For a generic congruence $Z$ the point-exponential map is stable so has folds, cuspidal edges, swallowtail points and the usual transverse self intersections. The folds (resp. cusps, swallowtails) correspond to contact of type $A_1$ (resp. $A_2, A_3$) between $Z$ and the set of lines through the focal point. 
}
\end{theo}
\begin{proof}
{\rm
\noindent (1), (2)  We have seen that the tangent space to the set of lines through a point $p$ corresponds to a plane determining a generators of the quadric $Q_L$ in the projectivised space. If we have tangency with the congruence than the sum of the tangent space to $Z$ and this plane has dimension $\le 3$, which projectivising means that the line determined by the tangent space to $Z$ and the generator meet, and by definition the corresponding point $x_0$ is a focal point of $Z$. The rest of the result follows from the transversality Proposition \ref{prop:sectransv}. We deduce that at most points on the surface the contact map is of type $(u,v^2)$, but we expect a curve of points where the contact with the surface is ${\cal K}$-equivalent to $(u,v^3)$ and isolated points where it is ${\cal K}$-equivalent to $(u,v^4)$. The corresponding subsets will be respectively smooth, cuspidal edge, swallowtail. We will also get the usual generic self-intersections. 

\noindent Part (3) is clear and (4) follows from (3) but is not hard to prove directly. 
}
\end{proof}

\begin{rems}
{\rm
\noindent (1) The discriminant of a germ of type $A_k$ is the same as that for an $A_k$ function: in $\BR^3$ the classical cuspidal edge, swallowtail etc. 

\noindent (2) In $1$-parameter families of congruences one expects to get isolated incidences where we have contact of type $(u,v^5)$ (referred to as a butterfly), and $(uv,u^2\pm v^2)$. 
}
\end{rems}

\subsection{Lines in a plane}\label{ssec:linesinaplane}

\begin{theo}\label{theo:linesinaplane}
{\rm
\noindent (1) The set of lines in a plane $P$ is tangent to the congruence $Z$ at $L$ if and only if $P$ is a focal plane of the congruence at $L$.

\noindent (2) For a generic congruence the contact at $L$ is of type $A_k$ for $0\le k\le 3$. These singularities are ${\cal K}$-versally unfolded within the family of lines through a point. 

\noindent (3) Taking $P=\{x_2=0\}$, working at $[e_3,0]$, parametrising $Z$ as $(u,v,b_1,b_2)$ the condition for singular $A_1$ contact is $b_{2u}=0$. The condition for $A_{\ge 2}$ contact, corresponding to the cuspidal edge on the plane focal set, is $b_{1u}=b_{2uu}=0$, for $A_{\ge 3}$ contact, corresponding to a swallowtail is $b_{2u}=b_{2uu}=b_{2uuu}=0$ at $(0,0)$. 

\noindent (4) For a generic congruence the dual of the point-focal set of $Z$ is the plane-focal set. More precisely the tangent plane to the point-focal set at a smooth point is the focal plane {\it for the other} focal point. 

\noindent (5) For a generic congruence $Z$ the exponential map is stable so has folds, cusps, swallowtail points and corresponding transverse self-intersections. 
}
\end{theo}

\begin{proof}
{\rm 
\noindent (1) This follows from the same argument as in the proof of Theorem \ref{theo:linesthroughapoint}; the tangent space to the set of lines in a plane $P$ at a point $L$ yields one of the generators of the quadric $Q_L$ and tangency between the congruence $Z$ at $L(z)$ and this tangent space means the line $l_z$ determined by the tangent space $T_zZ$ meets this generator, and we know this means $P$ is a focal-plane. 

\noindent Part (2) follows from the transversality result in Proposition \ref{prop:sectransv}. 

\noindent (3) The line determined by $(u,v,b_1,b_2)$ lies in the plane $x_2=0$ if and only if $v=b_2=0$, so the condition for a singularity is $b_{2u}=0$ and for degenerate contact $b_{2u}=b_{2uu}=0$. 

\noindent (4) Using the parametrisation at non-stall points, the plane $c_1x_1+c_2x_2+c_3x_3=d$ contains the tangent plane to the focal set if $(c_1,c_2,c_3)$ lies in the kernel of the derivative of the exponential map. This yields two equations involving $t$ and eliminating $t$ gives the condition on $c_1, c_2$ for a focal plane, with $c_3, d$ determined by $c_1u+c_2b+c_3=0, c_1b_1+c_2b_2=d$.  

\noindent (5) Follows as for the point exponential map. 
}
\end{proof}

\begin{rems}
{\rm
\noindent (1) On the point focal-set there will be cuspidal edges, double point curves and swallowtail points. Dualising these become parabolic curves, points of bitangency and cusps of Gauss on the plane-focal set, and vice-versa. We have seen that the condition for a cuspidal edge on the plane-focal set is $b_{2u}=b_{2uu}=0$ which means that we have a parabolic curve on the point-focal set, the further condition $b_{2uu}=0$ means we have a cusp of Gauss. (All partial derivatives evaluated at $(0,0)$ as usual.) 

\noindent (2)  It would be interesting to understand the way these two sets of objects appear on each focal set. This is done for normal congruences in \cite{wilk}. The following result provides some information in the general case. 
}
\end{rems} 

\begin{prop}\label{prop:FocalSurf_at_Par}
{\rm
For a generic congruence, the focal set is a regular surface along the image of the parabolic curve except at isolated points where the cuspidal edge cuts it (these points are precisely the folded singularities of the BDE (\ref{eq:BDEfocaldirec}) 
of the torsal directions; Proposition \ref{prop:flatdirections}). Every regular point on that curve is a hyperbolic point of the focal surface.
}
\end{prop}

\begin{proof}
We parametrise the congruence by $(u,v)\mapsto [(u,v, 1), (b_1,b_2,0)]$.  A point $(b_1,b_2,0)+t(u,v, 1)\in \mathbb R^3$ is on the focal set when $g(u,v,t)=0$ with 
\[
g(u,v,t)=
t^2+t(b_{1u}+b_{2v})+(b_{1u}b_{2v}-b_{1v}b_{2u}).
\]
The discriminant of $g=0$ viewed as a quadratic equation in $t$ is 
$(b_{2v}-b_{1u})^2+4b_{1v}b_{2u}$ which vanishes precisely on the parabolic curve.
\\
Suppose that $L(0,0)$ is on the parabolic curve and take $b_{1u}=b_{2u}=b_{2v}=0$ and $b_{1v}\ne 0$ at $(0,0)$ (see Proposition \ref{prop:transv}). Then  the double root of $g(0,0,t)=0$ is at $t=0$. For a generic congruence, $g_u(0,0,0)=-(b_{1v}b_{2uu})(0,0)\ne 0$ or $g_v(0,0,0)=-(b_{1v}b_{2uv})(0,0)\ne 0$; suppose the latter. Then by the implicit function theorem, the surface $g=0$ is locally at the origin the graph of 
 $v=h(u,t)$ for some smooth function $h$. The focal set is then locally parametrised by $\psi(u,t)=(b_1(u,h(u,t)),b_2(u,h(u,t)),0)+t(u,h(u,t), 1)$. We  have $h_t(0,0)=0$ and $h_u(0,0)=-(g_u/g_v)(0,0)=-(b_{2uu}/b_{2uv})(0,0)$. Therefore, $\psi_u(0,0)=((h_ub_{1v})(0,0),0,0)=(-(b_{1v}b_{2uu}/b_{2uv})(0,0),0,0)$ and $\psi_u(0,0)=(0,0,1)$. The two vectors are linearly independent if and only if $b_{2uu}(0,0)\ne 0$. For generic congruences $b_{2uu}(0,0)$ can vanish at isolated points on the parabolic curve. Such points are those where the unique (double) torsal direction is tangent to the parabolic curve, i.e., the folded singularities of the BDE (\ref{eq:BDEfocaldirec}), \cite{davbook}.
\\
With the above setting, the contact between the focal surface and its tangent plane $x_2=0$ is measured by the singularity type of the function 
$H(u,t)=b_2(u,h(u,t))+th(u,t)$  at the origin. We have $H_u(0,0)=H_{t}(0,0)=0$, $H_{tt}(0,0)=0$ and $H_{ut}(0,0)=-h_u(0,0)$. Thus, the singularity of $H$ at the origin is of type $A_1^-$ if the origin is a regular point of the focal surface (i.e., when $h_u(0,0)\ne 0$). It follows that all regular points on the image of the parabolic curve are hyperbolic points of the focal surface.
\end{proof}

\subsection{Lines in a given direction}

\begin{theo}\label{theo:linesinadirection}
{\rm
\noindent (1) The set of lines in a given direction $a$ is tangent to the congruence $Z$ at $L$ if and only if the direction map $m:U\to \BR\BP^2$ is singular.

\noindent (2) For a generic congruence the contact at $L$ is of type $A_k$, for $0\le k\le 2$. These singularities are ${\cal K}$-versally unfolded within the family of lines through a point. 

\noindent (3) For a generic congruence $Z$ the direction map $m:Z\to \BS^2$ is stable, that is has folds, cusps and transverse intersections of folds, with folds and cusps corresponding respectively to $A_1$ and $A_2$ singularities above. 
}
\end{theo}

\begin{proof}
{\rm
\noindent (1), (2) The set of lines in a given direction $a'$ are those of the form $[a',b]$, with $b$ arbitrary. As usual if we take $L$ to be $[e_3,0]$, and parametrise $\BL$ locally as $a=(a_1,a_2,1), b=(b_1,b_2,0)$ then $a'=e_3$ and the contact map for the congruence is $(a_1(u,v),a_2(u,v))$. This is singular when the direction map is singular. We deduce the remaining results from the transversality result.
\noindent (3) Aside from the condition $m\ne 0$ the map $m$ is arbitrary so the result is clear. 
}
\end{proof}

\subsection{Lines in parallel planes}

If $a\in \BS(\BR^{3*})$ then the planes $\a(x)=constant$ are parallel planes.

\begin{theo}\label{theo:linesothogtoadirection}
{\rm
\noindent (1) The set of lines in the planes $\a(x)=constant$ is tangent to the congruence $Z$ at $L$ if and only if the direction map at $L$  has a singularity, and this is degenerate when the direction map has a singularity worse than a fold, so generically a cusp.  

\noindent (2) For a generic congruence the contact at $L$ is of type $A_k$, for $0\le k\le 2$. These singularities are ${\cal K}$-versally unfolded within the family of lines through a point. 

\noindent (3) The discriminant of this family is dual to the discriminant of the lines in a given direction. Each smooth point of the critical set of the direction map has a tangent great circle and (one of) the poles traces out the discriminant of this family, as usual cusps corresponding to inflections, double points to bitangents. 
}
\end{theo}
\begin{proof}
{\rm
\noindent (1), (2) The line determined by $(a_1,a_2,b_1,b_2)$ is in one of the planes $e_1^*=const$ if and only if $a_1=0$. If we take $L$ to be $[e_3,0]$, and parametrise $Z$ locally as $a=(a_1,a_2,u,v)$ the condition for a singularity is $a_{1u}=a_{1v}=0$, and for degenerate contact in addition $a_{1uu}a_{1vv}-a_{1uv}^2=0$.  So we have a singular point of the direction map in the first case, and this is worse than a fold in the secodnd, and hence generically a cusp.  Part (3) is immediate. 
}
\end{proof}

\subsection{Lines in a plane through a point}

Here we are looking at contact of the congruence with curves so the relevant mappings are $\BR^2,0\to \BR^3,0$, again up to contact equivalence, in a $5$-parameter family. If we fix a point and a plane through that point, say $p=0, P=\{x_2=0\}$ then asking that a line in $\BL$ passes through that point and lies in that plane is imposing $3$ conditions on $Z$: taking the usual chart $[(a_1,a_2,1),(b_1,b_2,0)]$ this is $a_2=b_1=b_2=0$. The discriminant in this case is the set of pairs $(p,P)\in I(\BR^3,\BP)$ with $p\in L\subset \BP, L\in Z$.

\begin{theo}\label{theo:Linesinaplane}
{\rm
\noindent (1) The set of lines in a plane $P$ and through a point $p$ is tangent to the congruence $Z$ at $L$ if and only if $p$ is a focal point and $P$ a focal plane. 

\noindent (2) Locally the discriminant is smooth except along a codimension $2$ variety which correspond to the pairs $(p,P)$ where  $p$ a focal point, $P$ a focal plane.
Off a discrete set, a local model for the focal point and focal plane is a crosscap times $\BR^2$.

\noindent (3) At isolated points the discriminant has a model $\cal D$ given in the proof below; these correspond to points where we have both a degenerate focal point and focal plane. The subset of $A_1$ points is also given below and gives a surface with an isolated singularity. 
}
\end{theo}

\begin{proof}
{\rm
\noindent (1) Suppose that we are considering the point $p=0$ and the plane is $x_2=0$. The set of lines through $0$ in $x_2=0$ in the usual representation is given by $a_2=b_1=b_2=0$.  Working at a point $[e_3,0]$ the contact is given by the map $U\to \BR^3, (u,v)\mapsto (a_2,b_1,b_2)(u,v)$. As usual away from stall points $Z$ can be written locally as $(u,v,b_1,b_2)$ so we are considering $(v,b_1(u,v),b_2(u,v))$. The map has rank $\ge 1$ so we have map-germs of type $A_k$ that is ${\cal K}$-equivalent to $(u^{k+1},v,0)$ for some $k$. This germ has ${\cal K}$-codimension $2k+1$, so generically we will obtain singularities of type $A_0, A_1, A_2$. We have a singularity when $v=b_1=b_2=0=b_{1u}=b_{2u}=0$, but this is precisely the condition that $0$ is a focal point and $x_2=0$ a focal plane. 

\noindent (2), (3) Away from stall points the conditions for an $A_{\ge 2}$ are, in addition, $b_{juu}=0, j=1,2$. The set of lines through $p=0$ is given by $(b_1,b_2)=0$, and this mapping has $2$-jet 
\[
(b_{1v}v+B_1(u,v),b_{2v}v+B_2(u,v)), 
\]
with $B_1, B_2\in{\m}_2^2$. The two conditions clearly imply that the contact map has an $A_{\ge 2}$ singularity, that is the focal point is degenerate. The set of lines in $\{x_2=0\}$ on the other hand corresponds to $v=b_2(u,v)=0$, and $(v,b_{2uu})$ will also have an $A_{\ge 2}$ singular point. The converse is clear. 

The relevant discriminants are of codimension $1$ in the unfolding spaces: consider the versal unfoldings:
\[
\begin{array}{l}
A_1:(u,v^2+\a_0,\b_1v+\b_0),\\
{\cal D}= \{(\a_0,\b_1,\b_0)=(-v^2,\b_1,-\b_1v)\}, \, {\rm a\ crosscap}
\\
\\
A_2: (u,v^3+\a_1v+\a_0,\b_2v^2+\b_1v+\b_0), \\
{\cal D}=\{(\a_0,\a_1,\b_0,\b_1,\b_2)=(-v^3-\a_1v,\a_1,-\b_2v^2-\b_1v,\b_1,\b_2)\}.
\end{array}
\]
In the latter we have an $A_{\ge 1}$ point when $(\a_0,\a_1,\b_0,\b_1,\b_2)=(2v^3,-3v^2,\b_2v^2,-2\b_2v,\b_2)$.
}
\end{proof}

\subsection{Lines meeting a line}

The underlying geometry is determined by the incidence set
\[
I(L)=\{(L,L')\in \BL\times\BL:L\cap L'\ne \emptyset \ {\rm or\ parallel\ to\ }L\}.
\]
Taking the usual chart this is given by the set 
$$\{((\bar a,\bar b),(\bar a',\bar b')):a,a',b-b'\ {\rm linearly \ dependent}\},$$
that is,
\[
\begin{vmatrix}
a_1&a_2&1\\
a'_1&a'_2&1\\
b_1-b'_1&b_2-b'_2&0
\end{vmatrix}=0.
\]

Away from the diagonal this is a smooth $7$-dimensional manifold. Note that if $L'\ne L, L'\in I(L)$ then $I(L)\cap I(L')$ consists of the lines $L''$ in the plane spanned by $L,L'$.

\begin{theo}\label{theo:Linesmeetingaline}
{\rm
\noindent (1) The set $I(L)$ passes through $L(z)$ if and only if $L\in I(L(z))$, and, assuming $L\ne L(z)$, 
$I(L)$ is tangent to $Z$ at $L(z)$ if and only if $z$ is non-elliptic and $L$ is in the plane pencil of lines through one focal point and in the other focal plane.

\noindent (2) For generic $Z$ and for $L\ne L(z)$, the contact is of type $A_1, A_2, A_3, A_4$ or $D_4$ and the set $\{L\in \BL: I(L)\ {\rm is\ tangent\ to}\ Z\}$ is locally diffeomorphic to the corresponding discriminants. This set is the union of the focal pencils of lines of $Z$, that is the lines lying in the focal plane through the focal point. 

\noindent (3)  Parametrising $Z$ as usual at $[e_3,0]$ generically we have a $D_4$ if and only if choosing $b_{1u}=b_{2u}=0$ we have $b_{2uu}=b_{2uv}=0$.

\noindent (4) When $L=L(z)$ the contact of $I(L)$ with $Z$ is Morse unless $L$ is a hyperbolic point. 
}
\end{theo}

\begin{proof}
{\rm
\noindent Given $L\in \BL$ we know that $I(L)$ is smooth manifold apart from an $A_1$-singularity ar $L$ (Theorem \ref{theo:singI(L)}). So we have a $4$-parameter family of such hypersurfaces. Given a congruence $Z\subset \BL$ and a point $L(z)\in Z$ we need to consider two types of phenomena: (a) the contact beween $Z$ and $I(L(z))$ at $L(z)$; (b) the contact between $I(L)$ and $Z$ at $L(z)$ for other $L\in I(L(z))$. 

\noindent (1) We first consider the case when $L\ne L(z)=L_0$; we can apply the transversality result in this case because $I(\BL,\BL)$ has the necessary properties away from the diagonal. Clearly $I(L)$ passes through $L_0$ if and only if $L\in I(L_0)$ and the lines in $I(L)$ through $L_0$ are, as above, those lying in the plane spanned by $L_0$ and $L$, a $2$-dimensional family.   Again we fix our target line $L_0$ to be $[e_3,0]$. With the usual parametrisation $L$ corresponding to $(A_1,A_2,B_1,B_2)$ which meets $L_0$ if and only if $A_1B_2-A_2B_1=0$, and so corresponds to a point on the quadric $Q_{L_0}$. Note that the lines giving the same point are those in the plane spanned by $L_0, L$ through $L_0\cap L$. The lines $L'$ meeting $L$ are given by $(\bar a',\bar b')=(a_1',a_2',b_1',b_2')$ satisfy
\[
\begin{array}{rcl}
F(a_1',a_2',b_1',b_2')&=&(A_1-a_1')(B_2-b_2')+(A_2-a_2')(B_1-b_1')\\
&=&
(A_2B_1-A_1B_2)+(B_2a_1'-B_1a_2'-A_2b_1'+A_1b_2') +(a_1'b_2'-a_2'b_1')\\
&=&0.
\end{array}
\]
The tangent space to this set at $[e_3,0]$ is obtained by differentiating with respect to $a_1',a_2',b_1',b_2'$ and then setting $[a',b']=[e_3,0]$, so is the kernel of $(-B_2,B_1,A_2,-A_1)^t$, the polar plane of $[A_1:A_2:B_1:B_2]$. This point lies on the quadric $Q_{L_0}$ as $L_0\in I(L)$, so we have the tangent plane to the quadric at the corresponding point, which is spanned by the generators at the tangency points. This can also be deduced from Theorem 3.3 (6). So, to recap, there is a $3$-parameter family of points $L\ne [e_3,0]$ with $[e_3,0]\in I(L)$, each determines a point of the quadric, with each point on the quadric determining a plane pencil corresponding to the associated point and plane. The tangent space to $I(L)$ at $L_0$ corresponds to the tangent space to the quadric at the point.

Now suppose given a congruence $Z$; at each point $z$ there is a line $\ell_z$ in $\BP(T_{L(z)}\BL)$ determined by $T_zZ$. The manifolds $Z$ and $I(L')$, with $L'\in I(L)$, are tangent if and only if $\ell_z$ is contained in the corresponding tangent plane to the quadric. The tangent plane to $Q_L$ contains two generators, which $\ell_z$ will generally meet at two points of $Q_L$. The generators correspond to a point $x$ and plane $P$ with $x\in L\subset P$. So $x$ is the focal point at one point of intersection and $P$ is the focal plane at the other. If $\ell _z$ passes through the point of tangency then it is tangent so $z$ is parabolic and the result still holds (repeated focal points and planes).

\noindent (2) This follows from the transversality result in Proposition \ref{prop:sectransv}.

\noindent (3) The $D_4$ (and worse) singularities are those of corank $2$. Substituting $(a_1',a_2',b_1',b_2')=(u,v,b_1,b_2)$ in the above we obtain (all evaluated at $(0,0)$) from $F=F_u=F_v=F_{uu}=F_{uv}=F_{vv}=0$:
\[
\begin{array}{l}
A_1B_2-A_2B_1=0,\\
B_2-A_2b_{1u}+A_1b_{2u}+b_2=B_1+A_2b_{1v}-A_1b_{2v}+b_1=0,\\
-A_2b_{1uu}+A_1b_{2uu}+2b_{2u}=-A_2b_{1uv}+A_1b_{2uv}-b_{1u}=-A_2b_{1vv}+A_1b_{2vv}-2b_{1v}=0.
\end{array}
\]
The second and third equations above determine $B_1, B_2$ in terms of $A_1, A_2$, and the first equation gives $A_1^2b_{2u}+A_1A_2(b_{2v}-b_{1u})-A_2^2b_{1v}=0$ at $(0,0)$, which implies that $(A_1,A_2)$ is in a torsal direction. So we are at a non-elliptic point and, as usual can supppose that $b_{1u}=b_{2u}=0$, assume the torsal direction gives $A_2=0$, we also have $B_2=0$ and obtain, 
\[
\begin{pmatrix}b_{1uu}&b_{2uu}&0\\ b_{1uv}&b_{2uv}&0\\b_{1vv}&b_{2vv}&-2b_{1v}\end{pmatrix}\begin{pmatrix}0\\A_1\\1\end{pmatrix}=\begin{pmatrix}0\\0\\0 \end{pmatrix}.
\]
so $b_{2uu}=b_{2uv}=0$.  If $b_{2vv}\ne 0$ this determines $A_1$ and hence $B_1$.  So there is a unique corresponding $L$. 

\noindent (4) Now consider the case $L=L_0$, we cannot apply the transversality result here because the diagonal in $I(\BL,\BL)$ is not smooth. Recall that parametrising $\BL$ as $a=(a_1,a_2,1),\ b= (b_1,b_2,0)$ the set of lines meeting $L_0=(e_3,0)$ is given by $a_1b_2-a_2b_1=0$ and we need to consider the restriction of this function and its perturbations to $Z$ at $[e_3,0]$. Writing $a_1=u, a_2=v, b_i=b_{iu}u+b_{iv}v+O(2)$, we obtain a $2$-jet $F(u,v)=u^2b_{2u}+uv(b_{2v}-b_{1u})-v^2b_{1v}$ and this is not Morse when its discriminant vanishes, the condition for a parabolic point. Suppose now we are at a parabolic point; generically one expects the contact to be of type $A_2$ and at isolated points of type $A_3$. In fact, we get an $A_3$-singularity when $b_{2uu}=0$, equivalently, the singular set of the focal surface crosses the parabolic curve. It does not makes sense to consider ${\cal K}$-versality because the unfolding space is a cone. 
}
\end{proof}


\section {Middle surface}\label{sec:midSurf}

Recall that the middle surface is the locus of points $(p+p')/2$ in $\BR^3$ with $p, p'$ pais of (real or imaginary) focal points, given in Proposition \ref{prop:middlepointsComplex}. Over the non-elliptic set this is the locus of midpoints of the focal points.
	
\begin{prop}\label{prop:singmiddlesurface}
{\rm 
\noindent (1) The middle surface of a generic congruence can have only cross-cap singularities and these occur at isolated points on the image of the parabolic curve away from the singularities of the focal surface when $b_{1u}=b_{2u}=b_{2v}=b_{1uu}+b_{2uv}=0$; see \mbox{\rm Figure \ref{fig:cross.middlesurface}}. 

\noindent (2) The focal and middle surfaces are tangential along the image of the parabolic curve.

\noindent (3) All regular points on the image of the parabolic curve are hyperbolic points of the middle surface.
}
\end{prop}

\begin{figure}[htp]
	\begin{center}
		\includegraphics[scale=0.5]{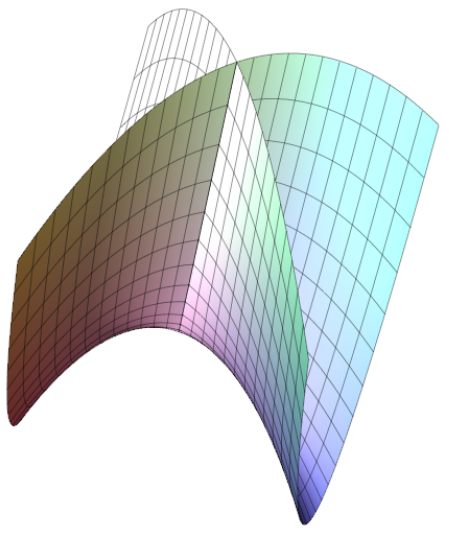}
		\caption{An example of a middle surface with a cross-cap singularity with 
			$b_1=2v + u^2 + 5vu - v^2 - 3u^3 - vu^2 + 7v^2u + v^3$ and $b_2=7u^2 - 2vu + 3v^2 + 5u^3 - 2vu^2 + 11v^2u + v^3$.}
		\label{fig:cross.middlesurface}
	\end{center}
\end{figure}
\begin{proof}
We parametrise the congruence in the usual way as $(u,v,b_1,b_2,)$.

\noindent (1) The middle surface is parametrised by 
\[
\eta(u,v)=(b_1(u,v),b_2(u,v),0)-\frac{1}{2}(b_{1u}(u,v)+b_{2v}(u,v))(u,v,1).
\]
The Jacobian matrix of $\eta$ at $(u,v)$ is 
\[
{\small
\left(
\begin{array}{ccc}
\frac{1}{2}(b_{1u}-b_{2v})-\frac12(b_{1uu}+b_{2uv})u&b_{2u}-\frac12 (b_{1uu}+b_{2uv})v&-\frac12 (b_{1uu}+b_{2uv})\\
b_{1v}-\frac12 (b_{1uv}+b_{2vv})u&\frac{1}{2}(b_{2v}-b_{1u})-\frac12(b_{1uv}+b_{2vv})v&-\frac12(b_{1uv}+b_{2vv})
\end{array}
\right).
}
\]
It has the same rank as the matrix 
\[
{\small
\left(
	\begin{array}{ccc}
		\frac{1}{2}(b_{1u}-b_{2v})&b_{2u}&-\frac12 (b_{1uu}+b_{2uv})\\
		b_{1v}&\frac{1}{2}(b_{2v}-b_{1u})&-\frac12(b_{1uv}+b_{2vv})
	\end{array}
	\right)
}
\]
which is singular if and only if its $2\times 2$ minors are zero. One of these minors  is $(b_{1u}-b_{2v})^2+4b_{1v}b_{2u}$ and vanishes precisely on the parabolic curve. For generic congruences the vanishing of two minors implies the vanishing of the third, so the singularities of the middle surface occur at isolated points on the parabolic curve.

We assume now that the origin is a parabolic point and take $b_{1u} =b_{2u}=b_{2v}=0$ and $b_{1v}\ne 0$ at $(0,0)$. Then the middle surface is singular at the origin if and only if $(b_{1uu} +b_{2uv})(0,0)=0$.
This condition is distinct from that for the focal surface to be singular (which is $b_{2uu}(0,0)=0$, see the proof of Proposition \ref{prop:FocalSurf_at_Par}).
\\
We write  $b_1=\sum_{k=1}^{3}\sum_{i=0}^kb^1_{ki}u^{k-i}v^i+O(4),$  $b_2=\sum_{k=1}^{3}\sum_{i=0}^kb^2_{ki}u^{k-i}v^i+O(4)$ 
and make changes of coordinates in the source and target. We find that the middle surface has a cross-cap singularity if and only if 
\[
(4b^1_{11}b^1_{31} + 4b^1_{11}b^2_{32} - (b^1_{21})^2 + 4(b^2_{22})^2)(3b^1_{11}b^1_{30} + 2b^1_{11}b^2_{20}+ b^1_{11}b^2_{31} - b^1_{20}b^1_{21} - 2b^1_{20}b^2_{22})\ne 0,
\]
The above condition is satisfied for generic congruences by Proposition \ref{prop:transv}.

\noindent (2) Taking $b_1,b_2$ as in Proposition \ref{prop:FocalSurf_at_Par} the tangent plane of the two surfaces at the origin is  $x_2=0$.

\noindent (3) With the above setting, the contact between the middle surface and its tangent plane $x_2=0$ at the origin is measured by the singularities of $H(u,v)=b_2(u,v)-\frac{1}{2}(b_{1u}(u,v)+b_{2v}(u,v))v$. This function has an $A_1^-$-singularity at the origin and the result follows.
\end{proof}
\begin{rem}
{\rm 
\noindent Away from the image of the parabolic curve, the focal and middle surfaces can have elliptic and parabolic points.
}
\end{rem}

\begin{acknow}
The work in this paper was partially supported by the FAPESP Thematic project grant 2019/07316-0.
\end{acknow}

\noindent 
JWB: Department of Mathematical Sciences,
The University of Liverpool,
Liverpool L69 3BX, UK.\\
E-mail: billbrucesingular@gmail.com\\

\noindent
FT: Instituto de Ciencias Matematicas e de Computacao - USP, Avenida Trabalhador sao-carlense, 400 - Centro,
CEP: 13566-590 - S\~ao Carlos - SP, Brazil.\\
E-mail: faridtari@icmc.usp.br

\end{document}